\documentclass[12pt, a4paper, twoside]{amsart}
\usepackage[T1]{fontenc}
\usepackage[utf8]{inputenc}
\usepackage{amsfonts}
\usepackage{amsmath}
\usepackage{amsthm}
\usepackage{bbm}
\usepackage{enumerate}
\usepackage{bm}
\usepackage{amssymb}

\usepackage{hyperref} %
\hypersetup{urlcolor=blue, colorlinks = true, linkcolor=blue, citecolor=blue}
\usepackage{mathtools}
\usepackage{wasysym}

\theoremstyle{definition}

\newtheorem*{claim}{Claim}
\theoremstyle{remark}

\theoremstyle{plain}
\newtheorem{thm}{Theorem}
\newtheorem{lem}{Lemma}

\makeatletter
\newcommand*{\house}[1]{%
	\mathord{%
		\mathpalette\@house{#1}%
	}%
}
\newcommand*{\@house}[2]{%
	\dimen@=\fontdimen8 %
	\ifx#1\scriptscriptstyle\scriptscriptfont
	\else\ifx#1\scriptstyle\scriptfont
	\else\textfont\fi\fi
	3 %
	\sbox0{%
		$#1%
		\vrule width\dimen@\relax
		\overline{%
			\kern2\dimen@
			\begingroup 
			#2%
			\endgroup
			\kern2\dimen@
		}%
		\vrule width\dimen@\relax
		\mathsurround=1.5\dimen@ 
		$%
	}%
	\ht0=\dimexpr\ht0-\dimen@\relax
	\dp0=\dimexpr\dp0+2\dimen@\relax
	\vbox{%
		\kern\dimen@ 
		\copy0 %
	}%
}

\newcounter{case}
\newcommand{\case}[1]{\vspace{8pt}\\\refstepcounter{case}\textbf{Case \arabic{case} (#1).}}
\newcommand{\C}{\mathbb{C}}
\newcommand{\N}{\mathbb{N}}
\newcommand{\Q}{\mathbb{Q}}

\title{Algebraic degree of series of reciprocal algebraic integers}

\author{Mathias L{\o}kkegaard Laursen}
\address{
	M. L. Laursen, 
	Department of Mathematics, 
	Aarhus University, 
	Ny Munkegade 118, 
	DK-8000 Aarhus C, 
	Denmark
}
\email{mll@math.au.dk}
\thanks{This research is supported by the Independent Research Fund Denmark}

\date{07, 11, 2022}

\begin{document}
	\begin{abstract}
		In this paper, I give sufficient conditions for any linear combination in $\Q$ of numbers $\sum_{n=1}^{\infty}\frac{b_{1,n}}{\alpha_{1,n}}$, $\ldots$, $\sum_{n=1}^{\infty}\frac{b_{K,n}}{\alpha_{K,n}}$ to have algebraic degree greater than an arbitrary fixed integer $D$ when the numbers $\alpha_{i,n}$ are algebraic integers of sufficiently rapidly increasing modulus and the $b_{i,n}$ are positive integers that are not too large.
	\end{abstract}

	\maketitle
\section{Introduction}
In 1975, Erd\H{o}s \cite{ErdosIrrOfInfSeries} gave a sufficient condition for an increasing series of reciprocal integers to be irrational, as stated in the below theorem.
\begin{thm}[Erd\H{o}s]
	\label{thm:Erdos}
	Let $\{a_n\}_{n\in\N}$ be an increasing sequence of natural numbers such that $a_n>n^{1+\varepsilon}$ for some $\varepsilon>0$ for all $n$ sufficiently large.
	\\
	If $\limsup_{n\to\infty} a_n^{1/2^n} = \infty$, then $\sum_{n=1}^{\infty}\frac{1}{a_n}$ is irrational.
\end{thm}

Later, in \cite{HanclCriterion}, Han\v{c}l extended this theorem to also cover the case of 
\begin{equation*}
	1\leq \limsup_{n\to\infty}a_n^{1/2^n}< \limsup_{n\to\infty}a_n^{1/2^n}<\infty,
\end{equation*} 
while also providing a related condition for finitely many series of fractions $\sum_{n=1}^{\infty} \frac{b_{i,n}}{a_{i,n}}$ ($i=1,\ldots, K$) with sufficiently small and positive $b_{i,n}$ to be irrational and linear independent over $\Q$.
\begin{thm}[Han\v{c}l]
	\label{thm:Hancl}
	Let $K\in\N$, and let $A_1,A_2,a,\varepsilon>0$ be real numbers such that $a<1\leq A_1<A_2$.
	For $i=1,\ldots,K$, let $\{a_{i,n}\}_{n\in\N}$ and $\{b_{i,n}\}_{n\in\N}$ each be sequences of natural numbers.
	Suppose that
	\begin{align*}
		\forall n\in\N:	\quad&	n^{1+\varepsilon} \leq a_{1,n} < a_{1,n+1}	
		,\\
		&	\limsup_{n\to\infty} {a_{1,n}}^{ 1/(K+1)^n } = A_2	
		,\\
		&	\liminf_{n\to\infty} {a_{1,n}}^{ 1/(K+1)^n } = A_1	
		,\\
		\forall n\in\N\ \forall 1\leq i\leq K:	\quad &	b_{i,n} < 2^{(\log_2 {a_{1,n}})^a}
		,\\
		\forall 1\leq i<j\leq K:	\quad & \lim_{n\to\infty} \frac{b_{i,n} a_{j,n}}{a_{i,n} b_{j,n}} = 0
		,\\
		\forall 1<i\leq K:	\quad &	{a_{i,n}}\ 2^{-(\log_2 {a_{1,n}})^a} < {a_{1,n}} < {a_{i,n}}\ 2^{(\log_2 {a_{1,n}})^a}
		.
	\end{align*}
	Then $\sum_{n=1}^{\infty} \frac{b_{1,n}}{a_{1,n}}, \ldots, \sum_{n=1}^{\infty} \frac{b_{K,n}}{a_{K,n}}$ are irrational and linearly independent over $\Q$.
\end{thm}
In 2017, Han\v{c}l and Nair \cite{Hancl-Nair} showed that integer sequences of the form $a_{n+1}=a_n^2 - a_n + 1$ with $a_1\geq 2$ will satisfy both $\sum_{n=1}^{\infty} 1/a_n \in\Q$ and 
\begin{align*}
	1 < \sqrt[4]{a_1^2 - a_1} \leq \liminf_{n\to\infty} a_n^{1/2^{n}} = \limsup_{n\to\infty} a_n^{1/2^{n}} < \infty,
\end{align*}
which exemplifies that the requirement $A_1<A_2$ cannot in general be omitted from Theorem \ref{thm:Hancl}.
The main result of \cite{Hancl-Nair} was a variant of Theorem \ref{thm:Erdos} where the $a_n$ may be square roots of positive integers when $\lim_{n\to\infty} a_n^{1/2^{n^2/2}}=\infty$.
Two years later, this result was generalised to the below theorem by Andersen and Kristensen \cite{Simon+SimonReciprocalAlgInts}, which gives a sufficient condition for $\sum_{n=1}^{\infty}1/\alpha_n$ to have large algebraic degree when $\alpha_n$ are algebraic integers of bounded degree.
\begin{thm}[Andersen and Kristensen] \label{thm:Simon}
	Let $d,D\in\N$, $\varepsilon>0$, and let $\{\alpha_n\}_{n\in\N}$ be a series of algebraic integers of maximal degree $d$ such that 
	\begin{align*}
		\forall n\in\N:	\quad &	n^{1+\varepsilon} \leq |\alpha_n|<|\alpha_{n+1}|
		,\\
		&\limsup_{n\to\infty}|\alpha_n|^{\frac{1}{ D^n\prod_{i=1}^{n-1}(d+d^i) }} = \infty
		,\\
		\forall n\in\N:	\quad &	\house{\alpha_n} = |\alpha_n|.
	\end{align*}
	Suppose that $\Re(\alpha_n)>0$ holds for all $n\in\N$ or that $\Im(\alpha_n)>0$ holds for all $n\in\N$.
	Then $\deg \sum_{n=1}^{\infty}\frac{1}{\alpha_n}$ is strictly greater than $D$.
\end{thm}
Both in the above theorem and for the remainder of this paper, $\house{\alpha}$ denotes the house of an algebraic number $\alpha$, which is defined as the maximum modulus among the conjugates of $\alpha$.

As Andersen and Kristensen note in their paper, their proof only really needs $\house{\alpha_n}$ to be bounded by $C|\alpha_n|$ for some uniform constant $C>0$.
Similarly, the restriction on the sign of the real (or imaginary) value of $\alpha_n$ is only to enforce that all $\alpha_n$ are contained in an open half plane not containing 0, which ensures that the partial sums $\sum_{n=1}^{N}\frac{1}{\alpha_n}$ are non-zero and non-conjugate to $\sum_{n=1}^{\infty}\frac{1}{\alpha_n}$ for sufficiently large $N$.

The main result of this paper is a generalisation of Theorem \ref{thm:Hancl} in the spirit of Theorem \ref{thm:Simon}, and the proof will combine the arguments used by the respective papers.
For the sake of clarity, we will, however, be slightly more explicit with the open half-plane containing all $\alpha_n$, as compared to Andersen's and Kristensen's proof of Theorem \ref{thm:Simon}.
For this purpose, I introduce notation $\Re_\zeta(z)$ to denote $\Re(\overline{\zeta} z)$ for $\zeta\in\C\setminus\{0\}$ and $z\in\C$, as $\Re_\zeta(z)>0$ is then equivalent to $z$ lying in the open half-plane with 0 on its border and moving in direction of $\zeta$.
\begin{thm}
	\label{conj problem 1}
	Let $D,K\in\N$, let $A_1,A_2,a,\varepsilon>0$ be real numbers such that $a<1\leq A_1<A_2$, and let $\zeta\in\C$ with $|\zeta|=1$.
	Let $\{d_n\}_{n\in\N}$ be a sequence of natural numbers, and write $D_n = \prod_{i=1}^{n}d_i$.
	For $i=1,\ldots,K$, let $\{b_{i,n}\}_{n\in\N}$ be a sequence of natural numbers, and let $\{\alpha_{i,n}\}_{n\in\N}$ be sequences of algebraic integers such that 
	\begin{align}
		\label{eq:increase}
		\forall n\in\N: 
		\ &|\alpha_{1,n}| < |\alpha_{1,n+1}|
		\\ \label{eq:a1n large}
		\forall n\in\N: 
		\ &|\alpha_{1,n}|\geq n^{1+\varepsilon}
		\\ \label{eq:deg bound}
		\forall n\in\N:
		\ &\left[\Q\left(\alpha_{1,n},\ldots, \alpha_{K,n}\right):\Q\right] \leq d_n
		\\ \label{eq:A1}
		&\liminf_{n\to\infty} |\alpha_{1,n}|^{ \frac{1}{D^n\prod_{i=1}^{n-1} (KD_i + d_i)} } = A_1
		\\ \label{eq:A2}
		&\limsup_{n\to\infty} |\alpha_{1,n}|^{ \frac{1}{D^n\prod_{i=1}^{n-1} (KD_i + d_i)} } = A_2 
		\\ \label{eq:house ain bound}
		\forall n\in\N\ \forall 1< i\leq K: 
		\ &2^{-(\log_2 |\alpha_{1,n}|)^a} < \frac{|\alpha_{1,n}|}{|\alpha_{i,n}|} < 2^{(\log_2 |\alpha_{1,n}|)^a} 
		\\ \label{eq:house}
		\forall n\in\N\ \forall 1\leq i\leq K: 
		\ & {b_{i,n}} \house{\alpha_{i,n}} \leq 2^{(\log_2 |\alpha_{1,n}|)^a} |\alpha_{i,n}|
		\\ \label{eq:Rezeta}
		\forall n\in\N\ \forall 1\leq i\leq K: 
		\ &\Re_\zeta\left(\frac{b_{i,n}}{\alpha_{i,n}}\right) > 0,
		\\ \label{eq:increasing fractions}
		\forall 1\leq i<j\leq K: 
		\ &\lim_{n\to\infty} \left(\Re_\zeta\left(\frac{b_{i,n}}{\alpha_{i,n}}\right) /\Re_\zeta\left(\frac{b_{j,n}}{\alpha_{j,n}}\right)\right) = 0.
	\end{align}
	Then $\deg\gamma>D$ when $\gamma$ is any non-trivial linear combination over $\Q$ of the numbers $\sum_{n=1}^{\infty} \frac{b_{1,n}}{\alpha_{1,n}}$, $\ldots$, $\sum_{n=1}^{\infty} \frac{b_{K,n}}{\alpha_{K,n}}$.
\end{thm}

\section{Auxiliary Results}
The proof of Theorem \ref{conj problem 1} will be split into two parts, the first of which will be based around the Weil height and the Mahler measure.
We recall the definitions below.

For $K$ being some finite field extension of $\Q$ of degree $d$, we define for $\alpha\in K$ the Weil height of $\alpha$ as the number 
\begin{align*}
	H(\alpha) := \prod_{\nu\in M_K} \max\{1,|\alpha|_\nu\}^{d_\nu/d},
\end{align*}
where $M_k$ denotes the set of places of $K$, and $d_\nu=[K_\nu:\Q_\nu]$ denotes the degree of the completion of $K$ with respect to place $\nu$ as an extension of the completion of $\Q$ with respect to $\nu$.
With the normalisation in the exponent $d_\nu/d$, the definition is independent of the field $K$ containing $\alpha$.
We define the Mahler measure of $\alpha$ as
\begin{align*}
	M(\alpha) := |a_d|\prod_{i=1}^{n} \max\{1,|\alpha_i|\},
\end{align*}
where $a_d$ here denotes leading coefficient of the minimal polynomial in $\mathbb{Z}[X]$ of $\alpha$, and $\alpha_1,\ldots,\alpha_d$ denote the conjugates of $\alpha$.

The proof will furthermore use the following lemmas, the first of which relates Weil height, Mahler measure, and house of algebraic integers.
The main part of the statement, $H(\alpha)=M(\alpha)^{1/d}$, is a classical result, which is presented in \cite{Waldshcmidt}.
The rest is essentially a trivial consideration, see \cite{Simon+SimonReciprocalAlgInts}.
\begin{lem}\label{LemmaAK3}
	Let $\alpha$ be an algebraic number of degree $d$.
	Then
	\begin{align*}
		H(\alpha) = M(\alpha)^{1/d} \leq \house{\alpha} \leq M(\alpha) = H(\alpha)^d
	\end{align*}
\end{lem}
The second lemma is a list of further classical results regarding the Weil height, see \cite{Waldshcmidt}.
\begin{lem}
\label{LemmaAK4}
	Let $\alpha, \beta$ be algebraic numbers.
	Then
	\begin{gather*}
		H\left(1/\alpha\right) = H(\alpha) \text{ if } \alpha\neq 0,
		\quad
		H(\alpha+\beta)\leq 2H(\alpha)H(\beta),
		\\
		H(\alpha\beta)\leq H(\alpha)H(\beta)
	\end{gather*}
\end{lem}
Similar results are likewise true for the degree function, as seen by the below lemma.
\begin{lem}
\label{LemmaAK5}
	Let $\alpha, \beta$ be algebraic numbers.
	Then
	\begin{gather*}
		\deg(1/\alpha) = \deg(\alpha)\text{ if } \alpha\neq 0,
		\\
		\deg(\alpha+\beta)\leq \deg(\alpha) \deg(\beta),
		\quad \deg(\alpha\beta) \leq \deg(\alpha) \deg(\beta)
	\end{gather*}
\end{lem}
This is essentially trivial:
Following the spirit of \cite{Isaacs}, the inequalities come from noting that $\alpha+\beta$ and $\alpha\beta$ both lie in the field extension $\Q(\alpha, \beta)$, which clearly has degree at most $\deg (\alpha) \deg (\beta)$ over $\Q$.
Noting $1/\alpha \in\Q(\alpha)$ and $\alpha\in\Q(1/\alpha)$ for $\alpha\neq 0$, it is likewise obvious that $\deg(1/\alpha)=\deg\alpha$.

The below lemma is central for the first part of the proof of Theorem \ref{conj problem 1}, and seems to originally be from \cite{Liouville}.
A proof may also be extracted from the proof of Theorem A.1 in Appendix A of \cite{Bugeaud}.
\begin{lem}
	\label{LemmaAK6}
	Let $\alpha, \beta$ be non-conjugate algebraic numbers.
	Then
	\begin{align*}
		|\alpha - \beta| \geq \frac{1}{2^{\deg(\alpha)\deg(\beta)} M(\alpha)^{\deg(\beta)} M(\beta)^{\deg(\alpha)}}
	\end{align*}
\end{lem}
In the second part of the proof of Theorem \ref{conj problem 1}, we will occasionally need the below simple estimate related to the exponent of the limes superior and limes inferior.
\begin{lem}\label{Lemma:Exponents inequality}
	Let $D,K,N\in\N$ be natural numbers, and let $\{d_n\}_{n\in\N}$ be a sequence of natural numbers.
	Writing $D_n=\prod_{i=1}^{n} d_i$, we have
	\begin{align*}
		D^{N+1}\prod_{i=1}^{N}(K D_i+d_i) \geq KD D_N \sum_{n=1}^{N} D^n\prod_{i=1}^{n-1}(KD_i+d_i).
	\end{align*}
\end{lem}
\begin{proof}
	The first statement proven by induction in $N$.
	Note that the statement clearly holds for $N=1$.
	Suppose it holds for $N-1$, for some $N>1$.
	Then
	\begin{align*}
		D^{N+1}\prod_{i=1}^{N}(KD_{i}+d_i) =\; &
		D(KD_{N}+d_N) \left(D^{N}\prod_{i=1}^{N-1}(KD_{i}+d_i)\right) 
		\\
		\geq\;&
		KDD_{N} \left( D^{N}\prod_{i=1}^{N-1}(KD_{i}+d_i) \right)
		\\& 
		+ d_N \left( KDD_{N-1} \sum_{n=1}^{N-1} D^n \prod_{i=1}^{n-1} (KD_{i}+d_i) \right)
		\\
		=\;& 
		KDD_{N}\sum_{n=1}^{N} D^n \prod_{i=1}^{n-1} (KD_{i}+d_i).
	\end{align*}
\end{proof}
Near the end of the proof, we will use a generalised version of a lemma from \cite{ErdosIrrOfInfSeries}, which Erd\H{o}s used for proving Theorem \ref{thm:Erdos}.
The current version is presented and proven in \cite{Simon+SimonReciprocalAlgInts}.
\begin{lem}
	\label{LemmaAK7}
	Let $\varepsilon>0$, and let $\{a_n\}_{n=1}^\infty$ be an increasing sequence of real numbers satisfying $a_n>n^{1+\varepsilon}$ for all $n\in\N$.
	Then
	\begin{align*}
		\sum_{n=N} \frac{1}{a_n} < \frac{2+ 1/\varepsilon}{a_N^{\varepsilon/(1+\varepsilon)}}.
	\end{align*}
\end{lem}

\section{Proof of Main Result}
The proof of Theorem \ref{conj problem 1} will be split into two lemmas:

\begin{lem} \label{Prop:eq AK2}
	Let $D,K\in\N$, let $\zeta\in\C$ with $|\zeta|=1$, and let $a, c,A_2>0$ such that $c<a<1<A_2$.
	Let $\{d_n\}_{n\in\N}$ be a sequence of natural numbers, and write $D_n = \prod_{i=1}^{n}d_i$.
	For $i=1,\ldots,K$, let $\{\alpha_{i,n}\}_{n\in\N}$ be a sequence of algebraic integers, and $\{b_{i,n}\}_{n\in\N}$ be a sequence of natural numbers.
	Suppose that equations \eqref{eq:deg bound},
	\eqref{eq:A2}, \eqref{eq:house ain bound}, \eqref{eq:house}, \eqref{eq:Rezeta}, \eqref{eq:increasing fractions} are satisfied, let $\beta_1,\ldots,\beta_K\in\mathbb{Z}$ be integers that are not all 0, and write
	\begin{align*}
		\gamma = \sum_{j=1}^{K}\beta_j\sum_{n=1}^{\infty}\frac{b_{j,n}}{\alpha_{j,n}},
		\qquad \gamma(N) = \sum_{j=1}^{K}\beta_j\sum_{n=N+1}^{\infty}\frac{b_{j,n}}{\alpha_{j,n}}
	\end{align*}
	If $\deg \gamma\leq D$ and $c\in(a,1)$, then
	\begin{align*}
		|\gamma(N)|\left(2^{D^{cN} \prod_{i=1}^{N-1}(KD_{i}+d_i)^{c}} \prod_{n=1}^{N} |\alpha_{1,n}|^{K} \right)^{D D_{N}} \geq 1
	\end{align*}
	holds for all sufficiently large $N$.
\end{lem}

\begin{lem} \label{Prop:eq AK2 negated}
	Let $D,K\in\N$, and let $A_1,A_2,a,\varepsilon>0$ be real numbers such that $a<1\leq A_1<A_2$.
	Let $\{d_n\}_{n\in\N}$ be a sequence of natural numbers, and write $D_n = \prod_{i=1}^{n}d_i$.
	For $i=1,\ldots,K$, let $\{\alpha_{i,n}\}_{n\in\N}$ and $\{b_{i,n}\}_{n\in\N}$ be sequences of complex numbers. 
	Suppose that equations 
	\eqref{eq:increase}, \eqref{eq:a1n large}, \eqref{eq:A1}, \eqref{eq:A2}, \eqref{eq:house ain bound} hold and that
	\begin{align}
		\label{eq:house bin bound} \forall n\in\N\ \forall 1\leq i\leq K:\quad |b_{i,n}| \leq 2^{(\log_2 |\alpha_{1,n}|)^a}.
	\end{align}
	Let $\beta_1,\ldots,\beta_K\in\mathbb{Z}$ be integers that are not all 0, and write
	\begin{align*}
		\gamma(N) = \sum_{j=1}^{K}\beta_j\sum_{n=N+1}^{\infty}\frac{b_{j,n}}{\alpha_{j,n}}
	\end{align*}
	Let $c\in(a,1)$. Then
	\begin{align*}
		\liminf_{N\to\infty} |\gamma(N)|\left(2^{D^{cN} \prod_{i=1}^{N-1}(KD_{i} + d_i)^{c}} \prod_{n=1}^{N} |\alpha_{1,n}|^{K} \right)^{D D_{N}} = 0.
	\end{align*} 
\end{lem}

One minor result that will be briefly used for proving both lemmas is that equation \eqref{eq:A2} implies
\begin{gather}
	\label{eq:Hancl11}
	|\alpha_{1,n}| \leq (2A_2)^{D^n\prod_{i=1}^{n-1} (KD_{i}+d_i)}
\end{gather}
for $n$ sufficiently large

We now prove Lemma \ref{Prop:eq AK2}:

\begin{proof}[Proof(Lemma \ref{Prop:eq AK2})]
	We introduce further notation
	\begin{gather*}
		\gamma_N :=  \sum_{i=1}^K \beta_i \sum_{n=1}^N \frac{b_{i,n}}{\alpha_{i,n}} ,
		\qquad
		\beta := \prod_{i=1}^{K} H(\beta_i).
	\end{gather*}
	By Lemma \ref{LemmaAK5} and equation \eqref{eq:deg bound}, we quickly find
	\begin{align*}
		\deg\gamma_N \leq \prod_{n=1}^{N} \deg\left(\sum_{i=1}^K \beta \frac{b_{i,n}}{\alpha_{i,n}}\right) 
		\leq D_N.
	\end{align*}
	Applying Lemma \ref{LemmaAK3} and Lemma \ref{LemmaAK4} followed by equation \eqref{eq:house}, we then get
	\begin{align*}
		M(\gamma_N) &= H(\gamma_N)^{\deg\gamma_N}
		\leq 
		\left(2^{NK} \prod_{i=1}^K H(\beta_i) \prod_{n=1}^{N} H(b_{i,n}) H\left(\frac{1}{\alpha_{i,n}}\right) \right)^{D_{N}}
		\\& 
		= 
		\left(2^{KN} \beta \prod_{i=1}^K \prod_{n=1}^{N} H(\alpha_{i,n})H(b_{i,n}) \right)^{D_{N}}
		\\&
		\leq \left(2^{KN} \beta \prod_{i=1}^K \prod_{n=1}^{N} \house{\alpha_{i,n}}\ \house{b_{i,n}} \right)^{D_{N}}
		\\&\leq \left(\beta 2^{KN + KN (\log_2 |\alpha_{1,N}|)^a}\prod_{i=1}^K \prod_{n=1}^{N} |\alpha_{i,n}| \right)^{D_N},
	\end{align*}
	using that $\alpha_{1,n}$ is non-decreasing and that $\house{b_{i,n}}=b_{i,n}$ as $b_{i,n}\in\N$.
	From equations \eqref{eq:house ain bound} and \eqref{eq:Hancl11}, we then have for $N$ sufficiently large that
	\begin{align}
		\nonumber
		M(\gamma_N)
		&< 
		\left(2^{2 KN \left(\log_2(2A_2) D^N \prod_{i=1}^{N-1}(KD_{i}+d_i)\right)^{a}} \prod_{n=1}^{N}|\alpha_{1,n}|^K\right)^{D_{N}}
		\\& \label{eq:MgammaN}
		\leq 
		\left(\frac{2^{D^{c N} \prod_{i=1}^{N-1}(KD_{i}+d_i)^{c}}}{2H(\gamma)} \prod_{n=1}^{N}|\alpha_{1,n}|^K\right)^{D_N}
		.
	\end{align}
	We now wish to apply Lemma \ref{LemmaAK7} to get an estimate on $|\gamma(N)|$.
	To do so, we need $\gamma\neq \gamma_N$, which is ensured for sufficiently large $N$ if the $\gamma_N$ are mutually distinct from a some point.
	\begin{claim}[$\gamma_M\neq\gamma_N$ for $M> N$ sufficiently large]
		To see this, let $R$ be the maximal value of $i$ such that $\beta_i\neq 0$ and assume without loss of generality that $\beta_R>0$ (otherwise replace each $\beta_i$ by $-\beta_i$ for all $i$).
		Using that $\Re_\zeta$ is clearly linear in $\mathbb{R}$, we find for each $n\in\N$ that
		\begin{align*}
			\Re_\zeta \left(\sum_{j=1}^{K} \beta_j \frac{b_{j,n}}{\alpha_{j,n}}\right) 
			&= \sum_{j=1}^{K} \beta_j \Re_\zeta\left( \frac{b_{j,n}}{\alpha_{j,n}}\right)
			\\&
			= \Re_\zeta \left( \frac{b_{R,n}}{\alpha_{R,n}} \right) \left(
			\beta_R + \sum_{j=1}^{R-1} \beta_j \frac{\Re_\zeta( b_{j,n}/ \alpha_{j,n} )}{\Re_\zeta( b_{R,n}/\alpha_{R,n} )} \right)
		\end{align*}
		Equation \eqref{eq:increasing fractions} then implies that for $n$ sufficiently large, we have
		\begin{align*}
			\left|
			\sum_{j=1}^{R-1} \beta_j \frac{\Re_\zeta( b_{j,n}/ \alpha_{j,n} )}{\Re_\zeta( b_{R,n}/\alpha_{R,n} )}
			\right|
			< \beta_R,
		\end{align*}
		and thus $\Re_\zeta \left(\sum_{j=1}^{K} \beta_j \frac{b_{j,n}}{\alpha_{j,n}}\right) > 0$, by equation \eqref{eq:Rezeta}.
		For $N$ sufficiently large, it hence follows that
		\begin{align*}
			\Re_\zeta \gamma_N = \sum_{n=1}^{N} \Re_\zeta \sum_{j=1}^{K} \beta_j \frac{b_{j,n}}{\alpha_{j,n}}
			< \sum_{n=1}^{M} \Re_\zeta \sum_{j=1}^{K} \beta_j \frac{b_{j,n}}{\alpha_{j,n}}
			= \Re_\zeta \gamma_M,
		\end{align*}
		which implies the claim.
	\end{claim}
	Since $\gamma$ can have at most $D$ conjugates, it then follows that $\gamma$ and $\gamma_N$ must be non-conjugate for $N$ sufficiently large, and we may apply Lemma \ref{LemmaAK6}, Lemma \ref{LemmaAK3}, and equation \eqref{eq:MgammaN} (in that order) to find
	\begin{align*}
		|\gamma(N)|&
		= |\gamma - \gamma_N|
		\geq \frac{1}{2^{\deg(\gamma)\deg(\gamma_N)} M(\gamma)^{\deg(\gamma_N)} M(\gamma_N)^{\deg(\gamma)}}
		\\&
		\geq
		\frac{1}{2^{DD_{N}} H(\gamma)^{DD_{N}}  M(\gamma_N)^{D}}
		\\&
		\geq \frac{1}{\left( 2^{D^{cN}\prod_{i=1}^{N-1}(KD_{i}+d_i)^{c}} \prod_{n=1}^{N} |\alpha_{1,n}|^{K}\right)^{D D_{N}}}.
	\end{align*}
	This proves the lemma.
\end{proof}

\begin{proof}[Proof (Lemma \ref{Prop:eq AK2 negated})]
	Applying equations \eqref{eq:house ain bound} and \eqref{eq:house bin bound} followed by \eqref{eq:increase}, we find
	\begin{align}
		\nonumber
		|\gamma(N)| &= \left| \sum_{n=N+1}^{\infty}\sum_{j=1}^{K} \frac{b_{j,n}}{\alpha_{j,n}} \right|
		\leq \sum_{n=N+1}^{\infty} \sum_{j=1}^{K} \frac{|b_{j,n}|}{|\alpha_{j,n}|}
		\\&\label{eq:AK3}
		\leq 
		\sum_{n=N+1}^{\infty} \frac{K 2^{(\log_2|\alpha_{1,n}|)^a} }{ |\alpha_{1,n}| 2^{-(\log_2|\alpha_{1,n}|)^a}}
		\leq 
		\sum_{n=N+1}^{\infty} \frac{2^{(\log_2 |\alpha_{1,n}|)^c} }{ |\alpha_{1,n}|},
	\end{align}
	for $N$ sufficiently large.
	
	We now split into two cases, both using the notation
	\begin{align*}
		a_n := |\alpha_{1,n}|,
		\qquad
		S_n := a_n^{\frac{1}{D^n \prod_{i=n}^{n-1}(KD_{i}+d_i)}}.
	\end{align*}
	\case{$a_n\geq 2^n$ for all $n$ sufficiently large}
		We continue on equation \eqref{eq:AK3} and use that the function $x^{(\log_2 x)^c}/x$ is decreasing for $x>1$ to find
		\begin{align}
			|\gamma(N)| &\nonumber
			\leq \sum_{N<n\leq \log a_{N+1}}^{\infty} \frac{2^{(\log_2 a_{n})^c} }{ a_{n}} + \sum_{n> \log a_{N+1}}^{\infty} \frac{2^{(\log_2 a_{n})^c} }{ a_{n}}
			\\&\nonumber
			\leq  \frac{2^{2(\log_2 a_{N+1})^c} }{ a_{N+1}} + \sum_{n> \log a_{N+1}}^{\infty} \frac{2^{(\log_2 2^n)^c} }{ 2^n }
			\\&\nonumber
			=  \frac{2^{2(\log_2 a_{N+1})^c} }{ a_{N+1}} + \sum_{n> \log a_{N+1}}^{\infty} \frac{1 }{ 2^{n-cn} }
			\\&
			\label{eq:Hancl14}
			\leq \frac{2^{2(\log_2 a_{N+1})^c} }{ a_{N+1}} + C\frac{1}{2^{\log_2 a_{N+1} - (\log_2 a_{N+1})^c}}
			\leq \frac{2^{(\log_2 a_{N+1})^\omega}}{a_{N+1}},
		\end{align}
		for sufficiently large $N$, where $C>0$ and $\omega\in(c,1)$ do not depend on $N$.
		The above equation is (safe for notational differences) a direct transcription of equation (14) of \cite{HanclCriterion}, which is repeated here for clarity.
	
		Next, we will make a choice of $N$ that will later show the conclusion of Lemma \ref{Prop:eq AK2 negated}.
		Let $\delta>0$ be a ``sufficiently'' small number (we will later make uniform assumptions on its size).
		By equations \eqref{eq:A2} and \eqref{eq:A1}, there exist $s_0\in N$ such that
		\begin{align}\label{eq:Hancl15}
			\max\{1, A_1-\delta\} < S_n < A_2 + \delta
		\end{align}
		holds for all $n\geq s_0$.
		For each such $s_0$, pick $s_1 \in \N$ minimal such that
		\begin{align}\label{eq:Hancl16}
			s_1 > D^{s_0}\prod_{i=1}^{s_0-1} (KD_{i}+d_i),
			\quad
			\max\{1,A_1-\delta\} < S_{s_1} < A_1 + \delta,
		\end{align}
		and pick then $s_2\in\N$ minimal such that
		\begin{align}\label{eq:Hancl17}
			s_2 > s_1,
			\quad
			A_2-\delta < S_{s_2} < A_2 + \delta.
		\end{align}
		For sufficiently large $s_0$, pick $N=N(s_0)\in\N$ minimal such that
		\begin{align}\label{eq:Hancl18}
			s_1\leq N<s_2,
			\quad
			S_{N+1} > \left(1 + \frac{1}{(N+1)^2}\right) \max_{s_1\leq j\leq N} \{S_j, A_2 - 2\delta\}.
		\end{align}
		This is doable as the contrary would imply
		\begin{align*}
			A_2 - \delta &< S_{s_2} \leq \left(1 + \frac{1}{s_2^2}\right) \max_{s_1\leq j<s_2} \{S_j, A_2 - 2\delta\}
			\\&
			\leq \ldots\leq \max \{S_{s_1}, A_2 - 2\delta\} \prod_{j=s_1+1}^{s_2}\left(1+\frac{1}{j^2}\right)
			\\&
			\leq (A_2 - 2\delta) \prod_{j=s_1+1}^{\infty}\left(1+\frac{1}{j^2}\right),
		\end{align*}
		which would be a contradiction for large enough $s_0$ (and thus $s_1$), regardless of $\delta$.
		We then apply equation \eqref{eq:Hancl18} along with Lemma \ref{Lemma:Exponents inequality} to find
		
		\begin{align}
			\nonumber
			a_{N+1} =&\; S_{N+1}^{D^{N+1}\prod_{i=1}^{N}(KD_i + d_i)} 
			\\\nonumber
			>& \left(1+\frac{1}{(N+1)^2}\right)^{D^{N+1}\prod_{i=1}^{N}(KD_i + d_i)}
			\ \max_{s_1\leq j\leq N}\{S_j, A_2-2\delta\}^{D^{N+1}\prod_{i=1}^{N}(KD_i + d_i)}
			\\\nonumber
			\geq& \left(1+\frac{1}{(N+1)^2}\right)^{D^{N+1}\prod_{i=1}^{N}(KD_i + d_i)}
			\\&\nonumber\ 
			\left(\max_{s_1\leq j\leq N}\{S_j, A_2-2\delta\}^{\sum_{n=1}^{N} D^{n}\prod_{i=1}^{n-1}(KD_i + d_i)}\right)^{KDD_{N}}
			\\\nonumber
			\geq& 
			\left(1+\frac{1}{(N+1)^2}\right)^{D^{N+1}\prod_{i=1}^{N}(KD_i + d_i)} \left(\prod_{n=s_1+1}^{N}{a_n}\right)^{KDD_N}
			\\&\nonumber \
			\left(
			\prod_{n=1}^{s_1}(A_2 - 2\delta)^{D^{n} \prod_{i=1}^{n-1} (KD_i + d_i)}
			\right)^{KDD_{N}}
			\\\nonumber \geq& \left(1+\frac{1}{(N+1)^2}\right)^{D^{N+1}\prod_{i=1}^{N}(KD_i + d_i)} \left(\prod_{n=1}^{N}{a_n}\right)^{KDD_N}
			\\&\label{eq:aN+1 bound case 1} \
			\left(\frac{1}{\prod_{n=1}^{s_0-1}{a_n}}
			\prod_{n=s_0}^{s_1}\frac{(A_2 - 2\delta)^{D^{n} \prod_{i=1}^{n-1} (KD_i + d_i)}}{{a_n}}
			\right)^{KD D_{N}},
		\end{align}
		for a small enough choice of $\delta$.
		For sufficiently large $s_0$, equation \eqref{eq:Hancl11} gives
		\begin{align}
			\label{eq:case1 1/prod a_n}
			\frac{1}{\prod_{n=1}^{s_0-1}{a_n}}
			&\geq \frac{1}{\prod_{n=1}^{s_0-1} (3A_2)^{D^{n} \prod_{i=1}^{n-1} (KD_i + d_i)}} 
			\geq (3A_2)^{-N},
		\end{align}
		by choice of $s_1$ and $N$.
		Meanwhile, equations \eqref{eq:Hancl15} and \eqref{eq:Hancl16} followed by Lemma \ref{Lemma:Exponents inequality} give
		\begin{align}
			\nonumber
			\prod_{n=s_0}^{s_1} &\frac{(A_2 - 2\delta)^{D^{n} \prod_{i=1}^{n-1} (KD_i + d_i)}}{{a_n}}
			\\&\nonumber
			\geq
			\left(\prod_{n=s_0}^{s_1-1}\frac{(A_2 - 2\delta)^{D^{n} \prod_{i=1}^{n-1} (KD_i + d_i)}}{(A_2+\delta)^{D^{n} \prod_{i=1}^{n-1} (KD_i + d_i)}}\right)
			\frac{(A_2 - 2\delta)^{D^{s_1} \prod_{i=1}^{s_1-1} (KD_i + d_i)}}{(A_1+\delta)^{D^{s_1} \prod_{i=1}^{s_1-1} (KD_i + d_i)}}
			\\ \label{eq:case1 prod A2-2delta / an}
			&\geq
			\prod_{n=s_0}^{s_1-1}\left(\frac{(A_2 - 2\delta)^2}{(A_2+\delta)(A_1+\delta)} \right)^{D^{n} \prod_{i=1}^{n-1} (KD_i + d_i)}
			\geq
			1,
		\end{align}
		by choosing $\delta>0$ small enough that $(A_2-2\delta)^2 > (A_2+\delta) (A_1 + \delta)$.
		Notice that since $d_i$ and $K$ are all positive integers, we must have $KD_i + d_i \geq 2$, which ensures
		\begin{align*}
			\prod_{i=1}^{N}(KD_i + d_i) &
			\geq \frac{\log 2}{\log\left(1+\frac{1}{(N+1)^2}\right)} N^3 D_N \prod_{i=1}^{N-1}(KD_i + d_i)^\omega,
		\end{align*}
		for large enough $N$ (recall $c<\omega<1$), using that $1/\log\big(1+\frac{1}{(N+1)^2}\big)$ is dominated by the polynomial $(N+1)^2$. Thus
		\begin{align}\label{eq:bound on power of 1/(+1 1/(1+N^2))}
			\left(1+\frac{1}{(N+1)^2}\right)^{D^{N+1}\prod_{i=1}^{N}(KD_i + d_i)}
			&\geq 2^{N^3 D^{N+1} D_N\prod_{i=1}^{N-1}(KD_i + d_i)^\omega}.
		\end{align}
		Applying this as well as equations \eqref{eq:case1 1/prod a_n} and \eqref{eq:case1 prod A2-2delta / an} to equation \eqref{eq:aN+1 bound case 1}, we have
		\begin{align*}
			a_{N+1} &
			\geq \left( \prod_{n=1}^{N}a_n \right)^{KDD_{N}} 
			2^{N^3 D^{N+1} D_N \prod_{i=1}^{N-1}(KD_i + d_i)^\omega} 
			(3A_2)^{-KNDD_{N}} 
			\\&
			\geq \left( \prod_{n=1}^{N}a_n \right)^{KDD_{N}} 2^{ N^2 D^{N+1}\prod_{i=1}^{N}(KD_i + d_i)^\omega},
		\end{align*}
		for $N(s_0)$ large enough, using that $(KD_N+d_N)/D_N\leq K+1$, and $K$ is constant.
		Recalling equations \eqref{eq:Hancl14} and \eqref{eq:Hancl11}, we now have
		\begin{align*}
			|\gamma(N)|&\leq \frac{2^{(\log_2 a_{N+1})^\omega}}{a_{N+1}}
			\leq \left( \prod_{n=1}^{N}{a_n} \right)^{-KDD_{N}} \frac{
				2^{\left(\log_2(2A_2) D^{N+1} \prod_{i=1}^{N} (KD_i+d_i)\right)^\omega }
			}{
				2^{ N^2 D^{N+1}\prod_{i=1}^{N}(KD_i + d_i)^\omega}
			}
			\\& \leq \left( \prod_{n=1}^{N}{a_n} \right)^{-KDD_{N}} 2^{-N D^{N+1} \prod_{i=1}^{N} (KD_i+d_i)^\omega},
		\end{align*}
		and so
		\begin{align*}
			|\gamma(N)|&\left(2^{\left(D^{N} \prod_{i=1}^{N-1} (KD_i+d_i)\right)^c} \prod_{n=1}^{N} \house{a_{1,n}}^{K} \right)^{D D_{N}}
			\\&
			\leq \frac{ 2^{D^{cN} \prod_{i=1}^{N-1} (KD_i+d_i)^{c}} }{ 2^{N D^{N} \prod_{i=1}^{N} (KD_i+d_i)^\omega}}
			\leq 2^{ -(K+1)^{\omega N}},
		\end{align*}
		for all sufficiently large $N(s_0)$.
		As this becomes arbitrarily small as $s_0$ tends to infinity, the lemma follows.
	\case{$a_n < 2^n$ infinitely often}
	Put $A = (1+A_2)/2$.
	By equation \eqref{eq:A2}, we may pick arbitrarily large $k\in\N$ such that
	\begin{align}\label{eq:Hancl22}
		S_k > A.
	\end{align}
	For each such $k$, pick $k_0\in\N$ maximal such that
	\begin{align}\label{eq:Hancl20}
		k_0 \leq k,
		\qquad a_{k_0} < 2^{k_0}.
	\end{align}
	Notice that the case assumption implies
	\begin{align}\label{eq:k0bounds}
		k_0 \underset{k\to\infty}{\longrightarrow} \infty.
	\end{align}
	As clearly $k_0<k$ for just slightly large $k$, pick $N\in\N$ minimal such that
	\begin{align}\label{eq:Hancl23}
		k_0\leq N<k,
		\qquad 
		S_{N+1} > \left(1+ \frac{1}{(N+1)^2}\right) \max_{k_0\leq j\leq N} S_j.
	\end{align}
	Such $N$ must exist as the contrary would imply
	\begin{align*}
		A &< S_k 
		\leq \left(1+ \frac{1}{k^2}\right) \max_{k_0\leq j<k} S_j
		\leq \cdots \leq S_{k_0}\prod_{j=k_0}^{k} \left(1+ \frac{1}{j^2}\right)
		\\&
		< S_{k_0} \prod_{j=k_0}^{\infty} \left(1+ \frac{1}{j^2}\right)
	\end{align*}
	for large enough $k$, as the number
	\begin{align*}
		C_k := S_{k_0} \prod_{j=k_0}^{\infty} \left(1+ \frac{1}{j^2}\right)
	\end{align*}
	tends to 1 as $k$ (and thus $k_0$, by \eqref{eq:k0bounds}) tends to infinity.
	Following the same argument, we may also conclude that $S_n < C_k$ for all $k_0\leq n\leq N$ when $k$ is sufficiently large.
	That leads to
	\begin{align*}
		\prod_{n=1}^{N}a_n &= \left(\prod_{n=1}^{k_0} a_n \right) \prod_{n=k_0+1}^{N} a_n
		< \left(\prod_{n=1}^{k_0}2^k\right) \prod_{n=k_0+1}^{N} C_k^{D^n\prod_{i=1}^{n-1}(KD_{i}+d_i)}
		\\&
		\leq 2^{k_0^2} C_k^{D^{N}\prod_{i=1}^{N-1}(KD_{i}+d_i)}
		\prod_{n=k_0+1}^{N-1} C_k^{D^n\prod_{i=1}^{n-1}(KD_{i}+d_i)},
	\end{align*}
	by using the choice of $k_0$ and equation \eqref{eq:increase}.
	Applying Lemma \ref{Lemma:Exponents inequality} and the lower bound on $N$, we reach
	\begin{align}\label{eq:AK12}
		\prod_{n=1}^{N} a_n &
		< 2^{N^2} \left(C_k^{D^{N}\prod_{i=1}^{N-1}(KD_{i}+d_i)}\right)^2
		= 2^{N^2} \left(C_k^2\right)^{D^{N}\prod_{i=1}^{N-1}(KD_{i}+d_i)}.
	\end{align}
	Aiming for a lower bound on $a_{N+1}$, we use equation \eqref{eq:Hancl23} and Lemma \ref{Lemma:Exponents inequality} to find
	\begin{align*}
		a_{N+1} =\; & S_{N+1}^{D^{N+1}\prod_{i=1}^{N}(KD_{i}+d_i)}
		\\
		>\; & \left( 1+\frac{1}{(N+1)^2} \right)^{D^{N+1}\prod_{i=1}^{N}(KD_{i}+d_i)}
		\\&
		\left( \max_{k_0\leq j\leq N} S_j \right)^{D^{N+1}\prod_{i=1}^{N}(KD_{i}+d_i)}
		\\
		\geq\; & \left(1+\frac{1}{(N+1)^2}\right)^{D^{N+1}\prod_{i=1}^{N}(KD_{i}+d_i)} 
		\\&
		\left(\max_{k_0\leq j\leq N} S_j\right)^{KD D_{N} \sum_{n=1}^{N} D^n \prod_{i=1}^{n-1}(KD_{i}+d_i)}
		\\
		\geq\; & \left(1+\frac{1}{(N+1)^2}\right)^{D^{N+1}\prod_{i=1}^{N}(KD_{i}+d_i)}
		\\&
		\left(\prod_{n=1}^{N} a_n\right)^{KD D_{N}}
		\left(\prod_{n=1}^{k_0-1} \frac{1}{a_n}\right)^{KD D_{N}}
		.
	\end{align*}
	By equation \eqref{eq:increase} and the choice of $k_0$, we have
	\begin{align*}
		\left(\prod_{n=1}^{k_0-1} \frac{1}{a_n}\right)^{KD D_{N}}
		&
		\geq \left(\prod_{n=1}^{k_0-1} 2^{-k_0}\right)^{KD D_{N}}
		\geq 2^{-K N^2 D D_{N}}.
	\end{align*}
	Recalling equation \eqref{eq:bound on power of 1/(+1 1/(1+N^2))} (which uses neither case assumption nor choice of $N$), we get for sufficiently large $N$ that
	\begin{align}
		\nonumber
		a_{N+1} &
		> 2^{N^3 D^{N+1} D_N\prod_{i=1}^{N-1} (KD_i+d_i)^\omega}
		\left(\prod_{n=1}^{N} a_n\right)^{KD D_{N}}
		2^{-KN^2 D D_{N}}
		\\&\label{eq:AK11}
		\geq 2^{N^2 D^{N+1} \prod_{i=1}^{N} (KD_i+d_i)^\omega}
		\left(\prod_{n=1}^{N} a_n\right)^{KD D_{N}}.
	\end{align}
	
	Repeating equation (20) of \cite{HanclCriterion}, we use that the function $2^{(\log x)^c}/x$ is decreasing combined with equation \eqref{eq:a1n large} to see that
	\begin{align*}
		\sum_{n=k}^\infty \frac{2^{(\log_2 a_{n})^c}}{a_n}
		&= \sum_{k\leq n\leq a_{k}^a} \frac{2^{(\log_2 a_{n})^c}}{a_n} + \sum_{n> a_{k}^a}^{\infty} \frac{2^{(\log_2 a_{n})^c}}{a_n}
		\\& 
		\leq a_{k}^a \frac{2^{(\log_2 a_{k})^c}}{a_{k}} + \sum_{n> a_{k}^a} \frac{2^{(\log_2 n^{1+\varepsilon})^c}}{n^{1+\varepsilon}}
		\\&
		\leq a_{k}^{(a-1)/2} + \sum_{n> a_{k}^a} \frac{1}{n^{1+\varepsilon/2}}
		\leq a_{k}^{(a-1)/2} + B_0 \frac{1}{(a_{k}^a)^{1+\varepsilon/2}}
		\\& 
		\leq a_{k}^{(a-1)/2} + a_{k}^{-a \varepsilon/3}
		\leq a_{k}^{-B},
	\end{align*}
	for $k$ sufficiently large, for some $0<B<1<B_0$ not depending on $k$.
	By equations \eqref{eq:AK3}, \eqref{eq:Hancl22} and \eqref{eq:AK11}, we then have
	\begin{align*}
		|\gamma(N)| &\leq \sum_{n=N+1}^{k-1} \frac{2^{(\log_2 a_n)^c}}{a_n} + \sum_{n=k}^{\infty} \frac{2^{(\log_2 a_n)^c}}{a_n}
		\leq \frac{2^{(\log_2 a_{N+1})^\omega}}{a_{N+1}} + a_{k}^{-B}
	\end{align*}
	Thus
	\begin{align*}
		|\gamma(N)|&\left(2^{D^{cN} \prod_{i=1}^{N-1} (KD_i+d_i)^{c}} \prod_{n=1}^{N} a_n^{K} \right)^{D D_{N}}
		\\&
		\leq \left(\frac{2^{(\log_2 a_{N+1})^\omega}}{a_{N+1}} + a_{k}^{-B}
		\right)
		\left(2^{D^{cN} \prod_{i=1}^{N-1} (KD_i+d_i)^{c}} \prod_{n=1}^{N} a_n^{K} \right)^{D D_{N}}
	\end{align*}
	It follows by $c<\omega$ and equations \eqref{eq:Hancl11} and \eqref{eq:AK11} that
	\begin{align*}
		\frac{2^{(\log_2 a_{N+1})^\omega}}{a_{N+1}} &
		\left(2^{D^{cN} \prod_{i=1}^{N-1} (KD_i+d_i)^{c}} \prod_{n=1}^{N} a_n^{K} \right)^{D D_{N}}
		\\&
		< 
		\frac{2^{
				(\log_2(2A_2)+1) 
				D^{\omega(N+1)} \prod_{i=1}^{N} (KD_i+d_i)^{\omega}
		}}{2^{
				N^2 D^{N+1} \prod_{i=1}^{N} (KD_i+d_i)^{\omega}
		}}
		< 2^{-(K+1)^{\omega N}},
	\end{align*}
	for sufficiently large $N$.
	Meanwhile, equations \eqref{eq:Hancl22} and \eqref{eq:AK12} imply that
	\begin{align*}
		a_{k}^{-B} &\left(2^{((K+1)DD_{N/2})^{cN}} \prod_{n=1}^{N} a_n^{K} \right)^{D D_{N}}
		\\&
		< \left(2^{D^{cN} \prod_{i=1}^{N-1} (KD_i+d_i)^{c}}  \right)^{D D_{N}} 
		\frac{\left(
			 2^{N^2} \left(C_k^2\right)^{D^{N}\prod_{i=1}^{N-1}(KD_{i}+d_i)} 
		\right)^{D D_{N}} }{ A^{B D^{k}\prod_{n=1}^{k-1} (KD_{i}+d_i)}}
		\\&
		= \left( 
			2^{N^2 + D^{cN} \prod_{i=1}^{N-1} (KD_i+d_i)^{c}}
		\right)^{D D_{N}}  \frac{
			\left(C_k^2\right)^{D^{N+1}D_{N}\prod_{i=1}^{N-1}(KD_{i}+d_i)}
		}{ (A^B)^{ D^{N+1}\prod_{n=1}^{N} (KD_{i}+d_i)}}
		\\&
		< 2^{D^{cN} \prod_{i=1}^{N} (KD_i+d_i)^{c}}
		\Big/ (A^{B/2})^{ D^{N+1}\prod_{n=1}^{N} (KD_{i}+d_i)}
		\leq 2^{-(K+1)^{cN}}
		,
	\end{align*}
	using that $C_k^2 < A^{B/2}$ for $k$ (and thus $N$) sufficiently large.
	For $k$ sufficiently large, we conlcule
	\begin{align*}
		|\gamma(N)| \left(2^{((K+1)DD_{N/2})^{cN}} \prod_{n=1}^{N} a_n^{K} \right)^{D D_{N}} 
		< 2^{-(K+1)^{\omega N}} + 2^{-(K+1)^{c N}},
	\end{align*}
	which clearly tends to 0 as $k$ (and thus $N$) grows large, and the lemma follows.
\end{proof}
\begin{proof}[Proof (Theorem \ref{conj problem 1})]
	It is clear that the entire hypothesis of Lemma \ref{Prop:eq AK2 negated} is implied by the hypothesis of Theorem \ref{conj problem 1}, as equation \eqref{eq:house} implies equation \eqref{eq:house bin bound}.
	It is likewise clear that the only part of the hypothesis of Lemma \ref{Prop:eq AK2} that is not also used in the hypothesis of Theorem \ref{conj problem 1} is the assumption that $\deg\gamma \leq D$.
	As the conclusions of the two lemmas are mutually exclusive, we conclude $\deg\gamma>D$.
\end{proof}
\section*{Concluding Remarks}
As in the case of Theorem \ref{thm:Simon}, the requirements using $\Re_\zeta$ (i.e. equations \eqref{eq:Rezeta} and \eqref{eq:increasing fractions}) are used solely to ensure that $\gamma_N$ is non-zero and non-conjugate to $\gamma$ for all sufficiently large $N$. Consequently, these requirements may be replaced any other set of conditions ensuring that property.
Note, however, that the property is required as one might otherwise construct a sequence converging to a rational number while satisfying all other parts of the hypothesis.

In the case of $K=1$, Theorem \ref{conj problem 1} implies
\begin{thm}\label{conj problem 1 K=1}
	Let $D\in\N$ be a natural number, let $\zeta\in\C$ with $|\zeta|=1$, and let $a,\varepsilon>0$ be real numbers. 
	Let $\{\alpha_{n}\}_{n\in\N}$ be a sequence of algebraic integers, and let $\{b_n\}_{n\in\N}$ be a sequences of rational integers.
	For $n\in\N$, write $d_n= \deg \alpha_n$ and $D_n = \prod_{i=1}^{n}d_i$. Suppose that
	\begin{align*}
	\forall n\in\N:
	\quad& n^{1+\varepsilon}\leq |\alpha_{n}| < |\alpha_{n+1}|
	,\\
	1\leq \liminf_{n\to\infty} |\alpha_{n}|^{ \frac{1}{D^n\prod_{i=1}^{n-1} (D_i + d_i)} } <
	&\limsup_{n\to\infty}|\alpha_{n}|^{ \frac{1}{D^n\prod_{i=1}^{n-1} (D_i + d_i)} } < \infty
	,\\
	\forall n\in\N: 
	\quad& b_{n} \house{\alpha_{n}} \leq 2^{(\log_2 |\alpha_{n}|)^a} |\alpha_{n}|
	,\\
	\forall n\in\N:
	\quad&\Re_\zeta({\alpha_{n}}) >0 
	.
	\end{align*}
	Then $\sum_{n=1}^{\infty} \frac{1}{\alpha_{n}}$ has algebraic degree strictly greater than $D$.
\end{thm}
By doing the right modifications to the proof of Theorem \ref{thm:Simon}, it may be improved so that the sequence $\{\alpha_n\}_{n\in\N}$ only needs to satisfy the hypothesis of Theorem \ref{conj problem 1 K=1} where the requirement 
\begin{equation*}
	1\leq \liminf_{n\to\infty} |\alpha_{n}|^{ \frac{1}{D^n\prod_{i=1}^{n-1} (D_i + d_i)} } <
	\limsup_{n\to\infty}|\alpha_{n}|^{ \frac{1}{D^n\prod_{i=1}^{n-1} (D_i + d_i)} } < \infty
\end{equation*}
is replaced by $\limsup_{n\to\infty}|\alpha_{n}|^{ \frac{1}{D^n\prod_{i=1}^{n-1} (D_i + d_i)} } = \infty$.
This will in particular remove the restriction that the $\alpha_n$ must be of bounded algebraic degree while also slacking the upper bound on $\house{\alpha_n}$.
\paragraph{\textbf{Acknowledgements}}\vspace{5 pt} I thank my supervisor Simon Kristensen for pointing me towards this problem, for helping me in finding the proper literature, and for advising me in the formulation of this paper.

\providecommand{\bysame}{\leavevmode\hbox to3em{\hrulefill}\thinspace}
\providecommand{\MR}{\relax\ifhmode\unskip\space\fi MR }
\providecommand{\MRhref}[2]{%
	\href{http://www.ams.org/mathscinet-getitem?mr=#1}{#2}
}
\providecommand{\href}[2]{#2}

\end{document}